\documentclass[letterpaper, 11pt]{amsart}
\usepackage[utf8]{inputenc}
\usepackage[T1]{fontenc}
\usepackage{mathtools}
\usepackage{amsmath}
\usepackage{amssymb}
\usepackage{yhmath}
\usepackage{graphicx}
\usepackage{mathrsfs}
\usepackage{bbm}
\usepackage{xcolor}
\usepackage{tikz-cd}
\usepackage{tikz}
\usetikzlibrary{patterns}
\usepackage{hyperref}

\usepackage{verbatim}
\usepackage{float}

\DeclareMathAlphabet{\mathpzc}{OT1}{pzc}{m}{it}

\usepackage{thmtools}
\usepackage{thm-restate}

\usepackage{caption}

\newtheorem{theorem}{Theorem}[section]

\newtheorem*{claim*}{Claim}

\newtheorem{lem}[theorem]{Lemma}

\newtheorem{cor}[theorem]{Corollary}

\newtheorem{prop}[theorem]{Proposition}

\theoremstyle{definition}

\newtheorem{Def}[theorem]{Definition}

\theoremstyle{remark}
\newtheorem{remark}[theorem]{Remark}

\newtheorem{Rmk}[theorem]{Remark}

\numberwithin{equation}{section}

\newcommand{\op}{\operatorname}

\newcommand{\be}{\begin{equation}}
\newcommand{\ee}{\end{equation}}
\newcommand{\Ga}{\Gamma}

\newcommand{\N}{\mathbb N}
\newcommand{\ga}{\gamma}

\newcommand{\La}{\Lambda}

\newcommand{\ba}{\backslash}

\newcommand{\cal}{\mathcal}
\newcommand{\br}{\mathbb R}
\newcommand{\SO}{\op{SO}}
\newcommand{\Sp}{\op{Sp}}

\newcommand{\bH}{\mathbb H}

\newcommand{\G}{\Gamma}

\renewcommand{\frak}{\mathfrak}

\newcommand{\e}{\varepsilon}

\newcommand{\fa}{\mathfrak a}

\newcommand{\id}{\op{id}}

\newcommand{\fg}{\frak g}

\newcommand{\Lie}{\op{Lie}}

\newcommand{\supp}{\op{supp}}

\renewcommand{\epsilon}{\e}

\newcommand{\SL}{\op{SL}}

\newcommand{\Hom}{\op{Hom}}
\newcommand{\bs}{\ba}

\title[Non-tempered subgroups]{Zariski-dense non-tempered subgroups in higher rank of nearly optimal growth}

\author{Miko{\l}aj Fr{^^c4^^85}czyk}
\address{Faculty of Mathematics and Computer Science, Jagiellonian University, ul. Łojasiewicza 6, 30-348 Krak{\'o}w, Poland}
\email{mikolaj.fraczyk@uj.edu.pl}

\author{Hee Oh}
\address{Department of Mathematics, Yale University, New Haven, CT}
\email{hee.oh@yale.edu}

\begin{document}

\begin{abstract}
We construct the first example of a Zariski-dense, discrete, non-lattice subgroup $\Ga_0$
of a higher rank simple
Lie group $G$, which is non-tempered in the sense that the quasi-regular representation
$L^2(\Ga_0\ba G)$ is non-tempered.

More precisely, let $n\ge 3$ and let $\Ga$ be the fundamental group of a closed hyperbolic $n$-manifold that contains a properly embedded totally geodesic hyperplane. We show that there exists a non-empty open subset $\cal O$ of $\op{Hom}(\Gamma, \SO(n,2))$ such that for any $\sigma\in \cal O$, the subgroup $\sigma(\Ga)$ is a  Zariski-dense and non-tempered Anosov subgroup of $\SO(n,2)$. In addition, the growth indicator of $\sigma(\Ga)$ is nearly optimal: 
it almost realizes the supremum of growth indicators among all non-lattice discrete subgroups, a bound imposed by property $(T)$ of $\SO(n,2)$.
\end{abstract}

\maketitle
\tableofcontents

\section{Introduction}
Let $G$ be a connected semisimple real algebraic group. 
Let $\Ga<G$ be a discrete  subgroup of $G$.  Denote by $dx$ a $G$-invariant measure on the homogeneous space $\Ga\ba G$.
Consider the Hilbert space $L^2(\Ga\ba G)=L^2(\Ga\ba G, dx)$.
The right translation action of $G$ on $\Ga\ba G$
induces a unitary representation of $G$ on $L^2(\Ga\ba G)$, called the quasi-regular representation.

A unitary representation  $(\pi,\cal H)$ of $G$ is called {\it tempered} if it is weakly contained in the (right) regular representation $L^2(G)$, 
i.e., any diagonal matrix coefficients of $(\pi,\cal H)$ can be approximated by a convex linear combination of diagonal matrix coefficients of $L^2(G)$, uniformly on compact subsets of $G$.
This notion, due to  Harish-Chandra, plays a central role in harmonic analysis on semisimple groups.

\begin{Def}\label{t1}
    We call a discrete subgroup $\Ga$ {\it tempered} in $G$ if  
    its quasi-regular representation $L^2(\Ga\ba G)$ is  tempered. 
\end{Def}

 Temperedness of $\Ga$ is equivalent to the statement that   all matrix coefficients of $L^2(\Ga\ba G)$ are $L^{2+\e}$-integrable  for any $\e>0$ \cite{CHH}. If $G$ has Kazhdan's property $(T)$, that is, all simple factors of $G$ have rank at least $2$ or are isogenous to $\Sp(n,1)$ or $F_4^{(-20)}$, then  a quantitative form of property $(T)$ implies the existence
 of $p=p_G>0$ such that for any non-lattice discrete subgroup $\Ga<G$,  all matrix coefficients of  $L^2(\Ga\ba G)$ are $L^p$-integrable (\cite{Co}, \cite{Oh}, \cite{kostant}). 

In rank-one groups, the situation is quite different, for example,
any lattice admits a non-elementary infinite index normal subgroup \cite{delzant}, whereas the Margulis normal subgroup theorem precludes such behavior in higher rank.
Moreover, there are also convex cocompact subgroups of $\SO(n,1)$, $n\ge 2$, whose critical exponents can be made arbitrarily close to the volume entropy of the hyperbolic $n$-space $\bH^n$, that is, $n-1$  \cite[Sec. 6]{Magee} (such
 examples cannot occur in higher rank because of
\eqref{ttt}). 
These high-exponent groups furnish Zariski-dense, non-tempered subgroups by \cite[Thm 1.4]{matsuzaki2020normalizer} and \cite[Thm 4.2]{corlette}.

For higher-rank groups, previously known non-tempered examples were all lattices of a proper algebraic subgroup of $G$
(\cite[Example B]{BLS}, \cite{BK}).
It remained open 
 whether one could find a {\it Zariski-dense},  non-lattice, non-tempered subgroup of a higher rank simple group $G$.
Our main result answers this in the affirmative:
\begin{theorem}\label{m1} For each $n\ge 3$, there exists a Zariski-dense,
non-lattice, non-tempered subgroup of $\SO(n,2)$.
\end{theorem}

\begin{remark}
For a geometrically finite discrete subgroup  $\Ga<\SO(n,1)$,
the hyperbolic manifold $\Ga\ba \bH^n$ possesses a square-integrable base eigenfunction of the Laplacian if and only if $\Ga$ is non-tempered (\cite{Pa}, \cite{Su}, \cite{Su2}).
By contrast, a recent result of \cite{EFLO} shows that for any non-lattice discrete subgroup $\Ga$ of a higher-rank simple algebraic group $G$, the base eigenfunction on the corresponding locally symmetric manifold is never square-integrable. Hence the appearance of a non-tempered subgroup in Theorem \ref{m1} underscores another sharp distinction in the behavior of infinite-volume locally symmetric manifolds between the higher rank and rank-one cases.
\end{remark}

Temperedness of $\Ga$ can be characterized in terms of its growth indicator $\psi_\Ga$.
Fix a Cartan decomposition $G=K\exp (\fa^+) K$, where
$K$ is a maximal compact subgroup and $\fa^+$ is a positive Weyl chamber of a Cartan subalgebra $\fa$. For $g\in G$, there exists a unique element $\mu(g)\in \fa^+$ such that $g\in K \exp \mu(g) K$, called the Cartan projection of $g$.

For a discrete subgroup $\Ga$ of $G$, denote by
$\mathcal L_\Ga\subset \fa^+$ its limit cone, which is defined as the asymptotic cone of $\mu(\Ga)$. The growth indicator $\psi_{G, \Ga}=\psi_\Ga :\fa^+\to \br \cup\{-\infty\}$, introduced by Quint \cite{Q1}, is a higher rank version of the critical exponent. It is $-\infty$ outside the limit cone $\cal L_\Ga$.
For each $v\in \cal L_\Ga$, the value $\psi_\Ga(v)$ encodes the exponential growth rate of $\Ga$ in the direction $v$:
\be\label{gi} \psi_\Ga(v)= \|v\|\cdot \inf_{v\in \cal C}\limsup_{T\to \infty}
\frac{\log\#\{\ga\in \Ga: \mu(\ga)\in \cal C, \|\mu(\ga)\|\le T\}}{T}  \ee 
where the infimum is taken over all open cones $\cal C\subset \fa^+$ containing $v$. This definition is independent of the choice of a norm $\|\cdot\|$ on $\fa$.

Denote by $\rho=\rho_G$ the half-sum of all positive roots of $(\op{Lie} G, \fa)$ counted with multiplicity.
The linear form $2\rho\in \fa^*$ gives the exponential volume growth rate of $G$: for any $v\in \fa^+$,
$$2\rho (v)= \|v\|\cdot \inf_{v\in \cal C}\limsup_{T\to \infty}
\frac{\log\op{Vol} \{g\in G: \mu(g)\in \cal C, \|\mu(\ga)\|\le T\}}{T}  $$
where the infimum is taken over all open cones $\cal C\subset \fa^+$ containing $v$. We have 
$$\psi_\Ga\le 2\rho\quad \text{on $\fa^+ $  }$$ for any discrete subgroup $\Ga<G$ and equality holds
for lattices $\Ga$ \cite{Q2}. If $G$ has Kazhdan's property $(T)$, there exists a constant $\eta_G>0$ such that
for any non-lattice discrete subgroup $\Ga$ of $G$, we have
\be\label{ttt} \psi_\Ga \le (2-\eta_G)\rho \quad \text{on $\fa^+ $  }\ee (\cite[Theorem 4.4]{Co}, \cite[Theorem 5.1]{Q3}, see also \cite[Theorem 7.1]{LO}).

\begin{Def}
  A discrete subgroup $\Ga<G$ has slow growth if $$\text{ $\psi_\Ga\le \rho\quad $ on $\fa^+$.}$$
\end{Def}
The slow growth means, informally, that the number of elements of $\Gamma$ in a ball of radius $R$ in $G$ is bounded (up to sub-exponential factors)
by a constant times the square root of the ball's volume
 as $R\to \infty$. 
It turns out that the slow growth property of $\Ga$ determines the temperedness: 
$$\text{$\psi_\Ga\le \rho $ on $\fa^+\quad $ if and only if $\quad \Ga$ is tempered.}$$
This was shown  in \cite{EO} for Borel-Anosov subgroups, and in \cite{LWW} for general discrete subgroups.

Theorem \ref{mm}, which is a more elaborate version of Theorem \ref{m1}, provides the first  Zariski-dense, non-lattice subgroups of
   higher rank simple Lie groups  that do not have
 slow growth. Moreover, these examples have nearly optimal growth.
For $n\ge 3$, the identity component of the
special orthogonal group $\SO^\circ(n,2)$ is a simple Lie group of rank two.
As discussed in  section \ref{sec4}, we can identify its positive Weyl chamber $\fa^+$  with $$\fa^+= \{v=(v_1, v_2, 0, \cdots,0, -v_2, -v_1)\in \br^{n+2}:v_1\ge v_2\ge 0\}.$$
The set of simple roots of $\SO^\circ(n,2)$ is given by $\alpha_1(v)=v_1-v_2$ and $\alpha_2(v)=v_2$, and $\rho$ is  the following: $$\rho (v) = 
 \frac{1}{2} \left( nv_1+ (n-2)v_2\right)$$
 for any $v=(v_1, v_2, 0, \cdots,0, -v_2, -v_1)\in \fa^+$.

\begin{theorem}\label{mm}
Let $n\ge 3$ and let $\Ga$ be the fundamental group of a closed hyperbolic $n$-manifold with a properly embedded totally geodesic hyperplane. For any $\e>0$, there exists a non-empty open subset $\cal O=\cal O(\e)$ of $\op{Hom}(\Ga, \SO^\circ(n,2))$ such that for any $\sigma\in \cal O$,
the following hold:
\begin{enumerate}
    \item $\sigma(\Ga)$ is a Zariski-dense, $\{\alpha_1\}$-Anosov\footnote{see Def. \ref{anosov} for the notion of an Anosov subgroup},  and non-tempered subgroup of $\SO^\circ(n,2)$ without slow growth;
\item for all $v\in \fa^+$, we have
$$\psi_{\sigma(\Ga)}(v)\le \left( \frac{2(n-1)}{n} +\e \right) \rho (v) ;$$
\item there exists a unit vector $v_\sigma  \in \fa^+$ such that
    \be\label{vs3} \psi_{\sigma(\Ga)}(v_\sigma)\ge \left( \frac{2(n-1)}{n} -\e \right)  \rho (v_\sigma) . \ee 
\end{enumerate}
 Moreover, $\sigma(\Ga)$ has nearly optimal growth in the sense that
 \be\label{vs0} \psi_{\sigma(\Gamma)}(v_\sigma)\ge  \sup_{\La} \psi_\Lambda(v_\sigma)-\varepsilon\ee 
  where the supremum is taken over \emph{all} non-lattice discrete subgroups $\La < \SO^\circ(n,2).$ 

\end{theorem}

We have an upper bound on the growth of arbitrary non-lattice discrete subgroups coming from the effective property (T) of $G$ (\cite{Oh}, see Proposition \ref{lo}). Inequalities \eqref{vs3} and \eqref{vs0} show that our examples
almost saturate this bound. At least inside $\SO(n,2)$, this means that one cannot  hope to improve existing growth-gap theorems (e.g. \cite{LO}) by merely imposing Zariski density. It remains an intriguing question whether such an improvement is possible in other higher rank groups, for example in $\SL_n(\mathbb R), n\geq 3$.

\begin{remark}
There are many examples of Zariski-dense discrete subgroups in higher rank that are tempered, for instance, the image of any  Hitchin representation of a surface group into a real split simple algebraic group of higher rank (\cite{EO}, \cite{DKO}). \end{remark}

Our construction of a non-tempered Zariski-dense subgroup of $\SO(n,2)$ goes as follows. We begin with a uniform lattice $\Ga$ in $\SO(n,1)$ that decomposes as an amalgamated product of two subgroups
over a uniform lattice in $\SO(n-1,1)$. For $n\ge 3$, any lattice
of $\SO(n,1)$  is non-tempered, when viewed inside $\SO(n,2)$ (Corollary \ref{ntt}). The inclusion $\op{id}_\Ga: \Ga\hookrightarrow \SO(n,2)$ can be deformed via the bending construction (\cite{JM}, \cite{Ka}), yielding a discrete Zariski-dense
 subgroup $\Gamma_1$ of $\SO(n,2)$. The heart of the paper is to show that $\Gamma_1$ is non-tempered. We present two proofs. In the first, we consider  the Chabauty topology on the space of closed subgroups of $\SO(n,2)$ and show that the property of being non-tempered is open, by studying the convergence of the matrix coefficients of quasi-regular representations\footnote{Fell's continuity of induction theorem \cite[Theorem 4.2]{Fell} yields a more general statement; we keep our explicit proof because it gives a slightly stronger result for $K$-finite matrix coefficients for semisimple real Lie groups.}. As a consequence, all sufficiently small (discrete) deformations of $\SO(n,1)$ remain non-tempered, so $\Gamma_1$ satisfies Theorem \ref{m1}.
 For the second proof, we track how  the growth indicator of $\Ga$ evolves under the deformation, using the property that $\Ga$ is an Anosov subgroup. The limit cone of the deformation is known to vary continuously in this  setting 
 (\cite{Ka}, see also \cite{DO}) and a certain critical exponent of $\Gamma_1$ varies continuously as well \cite{BCLS}.  Hence, for small deformations, the growth indicator of $\Gamma_1$ can be controlled by the growth indicator of $\Ga$ and hence it is not smaller than the half-sum of positive roots $\rho$, proving Theorem \ref{mm}.

\medskip 
\textbf{Acknowledgement.}
MF was supported by the Dioscuri programme initiated by the Max Planck Society, jointly managed with the National Science Centre in Poland, and mutually funded by the Polish Ministry of Education and Science and the German Federal Ministry of Education and Research. 
We would like to thank Dongryul Kim and Tobias Weich  for useful comments on a preliminary version of the article. We would also like to thank Marc Burger for
telling us about the reference \cite{Fell}. We also thank the referees for their careful reading of our paper and their useful comments.

\section{Convergence of matrix coefficients and Chabauty topology}\label{sec-mc}
Let $G$ be a locally compact second countable  group.
Let $\mathfrak C=\mathfrak C_G$ denote the space of all closed subgroups of $G$ equipped with the Chabauty topology, that is, a sequence of closed subgroups $H_n$ converges to $H$ as $n\to \infty$
if for any element $h\in H$, there exists a sequence $h_n\in H_n$ with $h_n\to h$ and the limit points of any sequence $g_{n}\in H_{n}$ belong to $H$. The space $\mathfrak C$ is a compact space.
When a sequence $H_i$ converges to a closed subgroup $H$, we say that $H$ is the Chabauty limit of $H_i$.
Note that the Chabauty limit of a sequence of discrete subgroups is not necessarily a discrete subgroup.

For a unimodular closed subgroup $H$ of $G$, denote by $\nu_H$ a Haar measure on $H$. For $s\in C_c(G)$ and any locally finite measure $\nu$ on $H$
we write $$\nu(s):=\int_H s(h) d\nu (h).$$
Note that for a non-negative function $s \in C_c(G)$ with $\nu_H(s)\ne 0$,
the normalized measure $\nu_H(s)^{-1}\nu_H$ is independent of the choice of a Haar measure $\nu_H$.
Let ${\mathcal M}(G)$ be the space of all locally finite Borel measures on $G$, equipped with the weak$^*$-topology. Throughout the paper, $e$ denotes the identity element of a relevant group.

\begin{prop} \label{ch} Let $\Gamma_n$ be a sequence of discrete subgroups of $G$ converging to a closed subgroup $H$ in the Chabauty topology.  
  Then $H$ is unimodular, and for any non-negative function $s\in C_c(G)$ with  $s(e)>0$,  we have
\be\label{de} \lim_{n\to \infty} \nu_{\Gamma_n} (s)^{-1} \nu_{\Gamma_n} =
\nu_H (s)^{-1}\nu_H \quad \text{ in $\mathcal M(G)$.  } \ee 
\end{prop}
\begin{proof}  Consider a non-negative function $s\in C_c(G)$ with  $s(e)>0$.
For simplicity, set $\nu_n=\nu_{\Gamma_n}$ and $\nu_n':=\nu_n(s)^{-1}\nu_n$.
Then $\nu_n'(s)=1$.

 First we show that the sequence $\nu_n'$ is relatively compact in $\mathcal M(G)$. Since $s(e)>0$, it follows from the continuity of $s$ that
 there exists a symmetric neighborhood $U$  of $e$ such that 
 $$\kappa:=\inf_{g\in U^2} s(g)>0.$$ 
 Fix any compact subset $C$ of $G$. Let 
\[m_C:=\max\{\# F\mid F\subset C, g_1U\cap g_2U=\emptyset \text{ for all } g_1\neq g_2\in F\}.\]
Note that
$$m_C\le \frac{\nu_G(CU)}{\nu_G(U)} .$$
For any $n\in\mathbb{N},$ choose a maximal subset 
$$F_n\subset \Gamma_n\cap C$$ such that $g_1U\cap g_2U=\emptyset$ for all $g_1\neq g_2\in F_n.$ Then $\Gamma_n\cap C\subset F_nU^2$, so
\[\nu_n(C)\leq \# F_n \cdot \nu_n(U^2) \leq \frac{m_C}{\kappa} \int s(g)d\nu_n(g).\]

Therefore for all $n\in \N$, we have $$\nu_n'(C) \le \frac{m_C}{\kappa} .$$ 
Since $C$ is an arbitrary compact subset of $G$, it follows that the sequence
$\nu_n'$, $n\in \N$, forms a relatively compact subset of $\mathcal M(G)$. 

Let $\nu\in \mathcal M(G)$ be a weak-* limit of the sequence $\nu_n'$. By construction, $\nu$ is a locally finite measure supported on $H$ and $\nu(s)=1$. 
It remains to show that $\nu$ is a Haar measure on $H$. Let $\varphi\in C_c(G)$ and $h\in H$. Let $\gamma_n\in \Gamma_n$ be a sequence with $\lim_{n\to\infty}\gamma_n=h$. Then, since $\nu_n'$ is a
Haar measure of $\Gamma_n$, we get 
\begin{multline*}\left|\int \varphi(g)-\varphi(hg)d\nu(g)\right|\leq \left|\int \varphi(g)d\nu(g)-\int \varphi(g)d\nu_n'(g)\right|\\+\left|\int \varphi(\gamma_n g)d\nu_n'(g)-\int\varphi(hg)d\nu_n'(g)\right| +\left|\int \varphi(hg)d\nu'_n(g)-\int\varphi(hg)d\nu(g)\right|.\end{multline*} 
The first and the third term converge to zero since $\nu_n'$ weakly converges to $\nu$. The middle term goes to zero because $\varphi(\gamma_n\cdot)$ converges uniformly to $\varphi(\cdot)$. Hence, the right hand side converges to $0$ as $n\to\infty$, so $\nu$ is indeed left $H$-invariant. Similarly, we can show that $\nu$ is also  a right $H$-invariant. This proves that $H$ is unimodular. Since $\nu(s)=1$, we have $\nu= \nu_H (s)^{-1}\nu_H$
and thus the desired convergence \eqref{de} follows from $\nu_n'\to \nu$.
\end{proof}

\begin{Rmk} The normalization of measures by the integral of $s$ is necessary in the above proposition. For example, if $G=\SL_2(\mathbb F_p((t)))$ and \[\Gamma_n:=\left\{\begin{pmatrix}
1 & f(t) \\
0 & 1
\end{pmatrix}\mid f(t) =a_nt^n+a_{n+1}t^{n+1}+\cdots + a_{2n}t^{2n}\in \mathbb F_p[t]  \right\},\]
then, as $n\to \infty$, $\Gamma_n$ converges to the trivial subgroup $\{e\}$ in the Chabauty topology, but the sequence $\nu_{\Gamma_n}$ of counting measures on $\Gamma_n$ fails to converge on the account of mass near identity blowing up to infinity. 
\end{Rmk}
On the other hand, we  can skip the normalization if the group $G$ has the no small subgroup property.
We say that a locally compact group $G$ has no small subgroup if
there exists a neighborhood of $e$ in $G$ which does not contain any non-trivial subgroup of $G$; this notion was first introduced in \cite{MY}. It is a well-known fact that
a real Lie group $G$ has no small subgroup; this can be easily seen,
using the fact that the exponential map is a diffeomorphism of a neighborhood of $0$ in $\frak g$ onto a neighborhood of the $e$ in $G$.

\begin{prop}\label{lem-wsc} Suppose that $G$ has the no small subgroup property (e.g., real Lie group).
Let $\Gamma_n$ be a sequence of discrete subgroups of $G$ which converges to a discrete subgroup $\Ga$ in the Chabauty topology.  Then as $n\to \infty$
$$\lim_{n\to \infty} \sum_{\gamma\in\Gamma_n}\delta_\gamma =
\sum_{\gamma\in\Gamma}\delta_\gamma  \quad\text{in $\cal M(G)$} $$
where $\delta_\ga$ denotes the Dirac measure at $\{\ga\}$.
\end{prop}

\begin{proof}
 Let  $\nu_n:=\sum_{\gamma\in\Gamma_n}\delta_\gamma$ and $\nu:=\sum_{\gamma\in\Gamma}\delta_\gamma$.
  Let $\varphi\in C_c(G)$.
 We need to show that 
$$ \lim_{n\to \infty} \int\varphi d\nu_n= \int\varphi d\nu .$$
 Let $\e>0$ be arbitrary.  Fix a compact subset $C\subset G$ and $\varphi\in C_c(G)$ supported on $C$.
 Enlarging $C$ if needed, we may assume that $\Gamma\cap \partial C=\emptyset.$
By the hypothesis that $G$ has no small subgroup property, there is an open neighborhood $U\subset G$ of the identity $e$ which contains no non-trivial subgroup of $G.$
We choose an open symmetric neighborhood $U_1\subset G$ of $e$ such that 
\begin{enumerate}
    \item $U_1^2\subset U;$
    \item $\gamma U_1^5\subset C$ for all $\gamma\in \Ga\cap C$;
    \item the collection $\gamma U_1^5, \gamma\in \G\cap C$, are pairwise disjoint; 
    \item  for all $\gamma\in C\cap \Gamma$ and $u\in U_1$,
    $$|\varphi(\gamma)-\varphi(\gamma u)|\leq \frac{\e }{\# (\Gamma\cap C)}.$$
\end{enumerate}
Consider the following compact subset
 $$C_1:=C\setminus \bigcup_{\gamma\in\Gamma\cap C}\gamma U_1.$$

 Note that $\Ga\cap C_1 =\emptyset$. Since the sequence $\Gamma_n$ converges to $\Gamma$ in the Chabauty topology, we have $\Gamma_n\cap C_1 =\emptyset$ for all $n$ large enough. For each fixed $\ga\in \Ga\cap C$, there exists $n_0=n_0(\ga)\ge 1$ such that
 $$\Gamma_n\cap \gamma U_1\neq\emptyset\text{ and } \Gamma_n\cap C_1=\emptyset \quad\text{ for all  $n\ge n_0$. }$$
Since $\Ga\cap C$ is finite, we have $n_0:=\max\{n_0(\ga):\ga\in \Ga\cap C\}<\infty$.

On the other hand, we claim that for any $\ga\in C\cap \Ga$ and $n\ge 1$,
 $$\# (\Gamma_n\cap \gamma U_1)\le 1.$$ Indeed, suppose there exists some element $\gamma_n\in\Gamma_n\cap \gamma U_1.$ Then 
\[\gamma_n^{-1}(\Gamma_n\cap \gamma U_1)=\Gamma_n\cap (\gamma_n^{-1}\gamma U_1)\subset \Gamma_n\cap U_1^2.\]
By the no-small-subgroups property of $G$, we have either $\Gamma_n\cap U_1^2=\{e\}$ or there is some element $\gamma_n'\in \Gamma_n\cap (U_1^4\setminus U_1^2)$; otherwise $\Gamma_n\cap U_1^2$ would be a non-trivial subgroup. In the second case, we would have 
$$\gamma_n\gamma_n' \in \gamma_n (U_1^4\setminus U_1^2)\subset \gamma U_1^5 \setminus \gamma U_1\subset C\setminus \gamma U_1.$$ Using property (3), we get $\gamma_n\gamma_n' \in C_1,$ contradicting the fact that $\Gamma_n\cap C_1=\emptyset.$ Therefore we must have $\Gamma_n\cap U_1^2=\{e\}$.
This implies that $\Gamma_n\cap \gamma U_1=\{\gamma_n\}$, proving the claim.

Therefore for all $\ga\in  \Ga\cap C$ and $n\ge n_0$, we have a unique element $\gamma_n\in \Gamma_n$
such that $\Gamma_n\cap \gamma U_1=\{\gamma_n\}$, and $\gamma_n \to \ga$ as $n\to \infty$.
Since
$$\int\varphi d\nu_n=\sum_{\gamma\in\Gamma\cap C}\varphi(\gamma_n)\quad\text{for all $n\ge n_0$},$$
we get from (4) that for all $n\ge n_0$,
\[\left|\int\varphi d\nu-\int\varphi d\nu_n\right| \leq \sum_{\gamma\in\Gamma\cap C}|\varphi(\gamma)-\varphi(\gamma_n)|\leq \varepsilon.\]
This finishes the proof.
\end{proof}

Let $G$ be unimodular and $dg$ a Haar measure on $G$.
For a closed unimodular subgroup $H$ of $G$,
there exists a unique $G$-invariant measure $d_{H\ba G}$ on $H\ba G$ such that for all $\psi\in C_c(G)$,
$$\int_G \psi dg=\int_{H\ba G}\int_H \psi(hg) d\nu_H(h) d_{H\ba G} (Hg). $$
We then have a unitary representation of $G$ on the Hilbert space
$$L^2(H\ba G)=\{f:H\ba G\to \br:\int_{H\ba G} |f|^2d_{H\ba G}<\infty\}$$ by right translations: $g.f(Hg'):=f(Hg'g)$ for $g,g'\in G$ and $f\in L^2(H\ba G)$.

\begin{prop}\label{l1} \label{app}
Let $\Gamma_n$ be a sequence of discrete subgroups of $G$ which converges to a closed  unimodular subgroup $H$ in the Chabauty topology. Let $K<G$ be a compact subgroup of $G$.

For any vectors $v, w \in L^2(H\bs G)$, there exist sequences $v_n, w_n \in L^2(\Gamma_{n}\bs G)$, $n\in \N$ such that  
\begin{enumerate}
    \item for all $g\in G$,
$$\lim_{n\to \infty} \langle v_n,g.w_n\rangle_{L^2(\Gamma_n\ba G)}=  \langle v, g.w\rangle_{L^2(H\ba G)}, $$
and the convergence is uniform on compact subsets of $G$;
 \item we have
$$\lim_{n\to \infty} \|v_n\|_{L^2(\Gamma_n\ba G)} =\|v\|_{L^2(H\ba G)}\;\; \& \;\; \lim_{n\to \infty} \|w_n\|_{L^2(\Gamma_n\ba G)} = \|w\|_{L^2(H\ba G)};$$
\item  we have that for all $n\in \mathbb N$, $$\dim \langle K.v_n \rangle \leq \dim \langle K.v\rangle \;\; \& \;\;
 \dim \langle K.w_n\rangle \leq \dim \langle K.w\rangle.$$
 \end{enumerate}
\end{prop}

\begin{proof}
Since $C_c(H\ba G)$ is dense in $L^2(H\ba G)$,
the matrix coefficient $g\mapsto \langle v, g.w\rangle_{L^2(H\ba G)}$ can be approximated by the matrix coefficients for continuous compactly supported functions, uniformly on compact subsets of $G$. This approximation can be done without increasing the dimensions of the spaces spanned by the $K$-orbits of $v$ and $w$. In fact, let $u_m$ be a sequence of compactly supported right $K$-invariant functions on $H\ba G$ converging to the constant function $1$ uniformly on compact subsets of $H\ba G$. Since the multiplication by $u_m$ is $K$-equivariant, we have $\dim \langle K.(u_m v)\rangle \leq \dim \langle K.v \rangle$, similarly for $w$. The matrix coefficients $g\mapsto \langle u_m v,g.u_m w\rangle_{L^2(H\ba G)}$ converge to $g\mapsto \langle v,g.w\rangle_{L^2(H\ba G)}$ uniformly on $G$. Thus we have shown  that $v,w$ can be replaced by compactly supported functions, spanning $K$-invariant subspaces of equal or smaller dimension. We need one more step to replace them by continuous functions. 

Let $\tilde \phi_m\in C_c(G)$ be a sequence of non-negative continuous functions with $\int \tilde \phi_m(g)dg=1$ and support contained in some neighborhood $\tilde U_m$ of $e$ such that $\tilde U_m\to \{e\}$ as $m\to \infty$.
Define $\phi_m\in C_c(G)$ by 
$$\phi_m(g)=\int_K \tilde \phi_m(k^{-1}gk)dk\quad\text{ for $g\in G$,}$$ where 
 $dk$ is the probability Haar measure on $K$. 
 Clearly, $\phi_m$ is non-negative, continuous and $\int \phi_m(g)dg=1$. The support of $\phi_m$ is contained in
 $U_m:=\{k\tilde U_m k^{-1}:k\in K\}$.
Note that $U_m\to \{e\}$ as $m\to \infty$; otherwise, we have, by passing to a subsequence,  $k_m g_m k_m^{-1}\to g$ for some  $k_m\in K$ converging to $k_0\in K$,  $g_m\in \tilde U_m$ and $g\ne e$.
Since $g_m \to e$ as $m\to \infty$, this is a contradiction.

 Consider the convolution $v\ast \phi_m$: 
 $$v\ast \phi_m (Hg)=\int_G v(Hgx) \phi_m(x^{-1}) dx\quad\text{ for $Hg\in H\ba G$}$$ and similarly for $w\ast \phi_m$.
 The functions $v\ast \phi_m$ and $w\ast \phi_m$ are continuous compactly supported functions on $H\ba G$.
 
Since the sequence $\phi_m$ is an approximate identity,
the matrix coefficient $g\mapsto\langle v\ast \phi_m, g. w\ast \phi_m\rangle_{L^2(H\ba G)}$ converges  to $g\mapsto\langle v, g. w\rangle_{L^2(H\ba G)}$, uniformly on compact sets. Furthermore, because $\phi_m$ is $K$-conjugation invariant,
 the map $v\mapsto v\ast \phi_m$ commutes with the action of $K$:
 $k.(v\ast \phi_m)= (k.v)\ast \phi_m$ for all $k\in K$.
 It follows that  
 $$\dim \langle K.(v\ast \phi_m)\rangle\leq  \dim \langle K.v\rangle,$$ and similarly for $w$. Therefore, we may assume  without loss of generality that $v, w \in C_c(H\bs G).$

 First, let $\tilde v_0\in C(G)$ be the lift of $v$ to $G$, i.e., for all $g\in G$,
 $\tilde v_0(g):=v(Hg)$. We note that 
$$\dim \langle K.v\rangle =\dim \langle K.\tilde v_0\rangle $$
 
Now, we choose a right $K$-invariant non-negative function $\varphi\in C_c(G)$ such that $\int_H \varphi(h g)d\nu_H(h) =1$ for every $g\in H\supp v\cup H\supp w$. 
 
Define $\tilde v \in C_c(G)$  by
$\tilde v(g):=\varphi(g)\tilde v_0(g)$ for all $g\in G$. Then for each $g\in G$, we have
\[ \int_H \tilde v(h g)d\nu_H(h) =v (g).\] Moreover $$\dim \langle K.\tilde v\rangle \leq \dim \langle K.\tilde v_0\rangle =\dim \langle K.v\rangle.$$

Choose a non-negative function $s\in C_c(G)$ such that $s(e)>0$ and 
$$\int_H s(h)d\nu_H(h)=1.$$  Set $\alpha_n:=\sum_{\gamma\in\Gamma_n}s(\gamma)$, and define $v_n\in C_c^\infty(\Gamma_n\bs G)$ as follows: for all $g\in G$,
 \[v_n(g) :=\alpha_n^{-1/2}\sum_{\gamma\in\Gamma_n}\tilde v(\gamma g) .\]
Then \[\dim \langle K.v_n\rangle \leq \dim\langle K.\tilde v\rangle \leq \dim \langle K.v \rangle.\]  

Let $\tilde w \in C_c(G)$ and $w_n \in C_c^\infty(\Gamma_n\ba G)$ be functions constructed in the same way for the vector $w$.

We claim that for all $g\in G$,
$$ \langle v_n, g. w_n\rangle_{L^2(\Gamma_n\ba G)} \to \langle v, g.w \rangle_{L^2(H\bs G)},$$ 
uniformly on compact subsets of $G$.
Indeed, 
\begin{align}\label{eq1} \langle v_n, g. w_n\rangle_{L^2(\Gamma_n\ba G)}
=&\alpha_n^{-1}\int_{\Gamma_n\bs G}\left(\sum_{\gamma\in\Gamma_n}\tilde v(\gamma x)\sum_{\gamma'\in\Gamma_n}\tilde w(\gamma' xg)\right)dx\\
=&\int_{G}\tilde v(x)\left(\alpha_n^{-1}\sum_{\gamma'\in\Gamma_n}\tilde w(\gamma' xg)\right)dx.\nonumber
\end{align}

Proposition \ref{ch} yields the weak-* convergence of measures $$\alpha_n^{-1}\sum_{\gamma'\in \Gamma_n}\delta_{\gamma'}\to d\nu_H.$$
It follows that 
\[\lim_{n\to \infty} \alpha_n^{-1}\sum_{\gamma'\in\Gamma_n}\tilde w(\gamma'x g)=\int_{H}\tilde w(h  xg)d\nu_H(h)\]  and the convergence is uniform for all $g$ and $x$ in a given compact subset of $G$. Indeed, for $x,g\in C,$ $C$ compact, the family of functions $\tilde w(\cdot xg)$ is equicontinuous and supported in a single compact set, so the integrals converge uniformly for any weak-* convergent sequence of measures. Since $\tilde v$ is compactly supported, we get 
\begin{align*}\lim_{n\to\infty}\langle v_n, g. w_n\rangle_{L^2(\Gamma_n\ba G)}  =&\int_G\tilde v(x)\int_\Gamma \tilde w(hx g)d\nu_H(h) dx ,\end{align*}
and the convergence is uniform for all $g$ in a given compact subset of $G$.
Since
\begin{align*} &\int_G\tilde v(x)\int_\Gamma \tilde w(h x g)d\nu_H(h) dg  =\int_G \tilde v(x) w(H x g)dx \\ & = \int_{H\bs G}v(H x) w(Hxg)d_{H\ba G} (Hx)=\langle v,g. w\rangle_{L^2(H\ba G)},
\end{align*}
this finishes the proof of (1) and (3). The claim (2) follows since the above argument applies when $v=w$ and $g=e$ and hence
gives $ \langle v_n, v_n\rangle_{L^2(\Gamma_n\ba G)} \to \langle v, v \rangle_{L^2(H\bs G)}$ and similarly for $w_n$ and $w$.
\end{proof}

\begin{remark}
This proposition implies that if $\Gamma_n$ converges to $H$ in the Chabauty topology,
then $L^2(H\bs G)$ is weakly contained in $\bigoplus_{n=n_0}^\infty L^2(\Gamma_n\bs G)$ for all $n_0\ge 1.$\footnote{Since submitting this paper, we have learned that this conclusion already follows from \cite[Theorem 4.2]{Fell}.}
\end{remark}

\section{Temperedness is a closed condition in $\Hom(\Ga, G)$}\label{sec-tempcl}
Let $G$ be a connected semisimple real algebraic group.
Let $P$ be a minimal parabolic subgroup of $G$ with a fixed Langlands decomposition $P=MAN$ where $A$ is a maximal real split torus of $G$, $M$ is the maximal compact subgroup of $P$, which commutes with $A$, and $N$ is the unipotent radical of $P$. We denote by $\fg$ and $\fa$ the Lie algebras of $G$ and $A$ respectively.
We fix a positive Weyl chamber $\fa^+\subset \fa$ so that $\Lie N$ consists of positive root subspaces.
Let $\Sigma^+=\Sigma^+(\fg, \fa)$ denote the set of all  positive roots for $(\fg, \fa^+)$. For each $\alpha\in \Sigma^+$, let $m(\alpha)$ be its multiplicity. We also write $\Pi\subset \Sigma^+$ for the set of all simple roots.
We denote by \be\label{rho} \rho=\frac{1}{2}\sum_{\alpha\in \Sigma^+} m(\alpha) \alpha\ee  the half sum of the positive roots for $(\mathfrak g, \mathfrak a^+)$, counted with multiplicity.

We fix a maximal compact subgroup $K$ of $G$ so that the
Cartan decomposition $G=K (\exp \fa^+) K$ holds, that is, for any $g\in G$, there exists a unique element $\mu(g)\in \fa^+$ such that $g\in K \exp \mu(g) K$.

Let $dg$ be a Haar measure on $G$.
The right translation action of $G$ on itself induces the regular representation $L^2(G)=L^2(G, dg)$.

Following Harish-Chandra, we call a unitary representation $(\pi,\cal H) $ of $G$ {\it tempered} if $\pi$ is weakly contained in the regular representation $L^2(G)$.

For any $p>0$, a unitary representation $(\pi,\cal H) $ of $G$ is said to be almost $L^p$-integrable if
all of its matrix coefficients are $L^{p+\e}$-integrable for any $\e>0$.

Denote by $\Xi=\Xi_G$ the Harish-Chandra function of $G$. It  is a bi-$K$-invariant function satisfying that
 for any $\e>0$, there exist $c, c_\e>0$ such that
$$c e^{- \rho (v)} \le \Xi(\exp v)\le c_\e e^{-(1-\epsilon) \rho (v)}\quad\text{ for all $v\in \fa^+$.}$$
We will use the following characterization of a tempered representation of $G$ given by Cowling, Haggerup and Howe:
\begin{theorem}\cite{CHH}  \label{chh}  For a unitary representation $(\pi,\cal H)$ of $G$,
the following are equivalent:
\begin{enumerate}
    \item $\pi$ is  tempered;
    \item $\pi$ is almost $L^2$-integrable;
\item for any $K$-finite unit vectors
$v_1, v_2\in \cal H$ and any $g\in G$,
$$|\langle \pi(g) v_1, v_2\rangle |\le \left(\op{dim}\langle \pi(K)v_1 \rangle\cdot  \op{dim}\langle \pi(K)v_2  \rangle \right)^{1/2}  \Xi_G(g). $$
\end{enumerate}
\end{theorem}

\begin{Def}\label{t2}
    We say that a unimodular subgroup $H$ is a {\it tempered} subgroup of $G$ (or $G$-tempered) if  
   the quasi-regular representation $L^2(H\ba G)$ is a tempered representation of $G$. 
\end{Def}
\begin{lem}\cite[Proposition 3.1]{BK} \label{bk} Let  $H$ be a unimodular closed subgroup of $G$. 
    If $H$ is $G$-tempered, then any unimodular closed subgroup $H'<H$ is also $G$-tempered.
\end{lem}

We show that temperedness is a closed condition both for the Chabauty topology and the algebraic topology (Theorems \ref{cl} and \ref{al}).
\begin{theorem}\label{cl}
 The Chabauty limit of a sequence of tempered discrete subgroups of $G$ is unimodular and
    tempered.
\end{theorem}

\begin{proof}
Suppose that $\Gamma_n$ is a sequence of tempered discrete subgroups converging to a closed
subgroup $H$ in the Chabauty topology. We have $H$ unimodular by Proposition \ref{ch}.
We claim that $L^2(H\ba G)$ is tempered. Suppose not.  
By Theorem \ref{chh},
there exist $K$-finite unit vectors $v,w\in L^2(H\ba G)$ and $g\in G$ such that
\begin{equation}\label{eq2}
|\langle v,  g. w\rangle_{L^2(H\ba G)} |>\Xi(g) \dim\langle K.v\rangle^{1/2}\dim\langle K.w\rangle^{1/2}.
\end{equation}
By Proposition \ref{app}, there exist vectors
$v_n,w_n\in L^2(\Gamma_n\bs G)$ such that  $\|v_n\|\to \|v\|, \|w_n\|\to\|w\|$ as $n\to\infty$, $\dim \langle K.v_n\rangle\leq \dim \langle K.v\rangle, \dim \langle K.w_n\rangle \leq \dim \langle K.w\rangle$, and
 $\langle v, g. w \rangle_{L^2(H\ba G)} =\lim_{n\to \infty}
 \langle v_n, g. w_n \rangle_{L^2(\Gamma_n\ba G)} $. We can normalize $v_n,w_n$ to be unit vectors without affecting the above properties.
 We deduce that for all $n$ large enough,
 $$|\langle v_n, g. w_n \rangle_{L^2(\Gamma_n\ba G)}| >\Xi(g) \dim\langle K.v_n\rangle^{1/2}\dim\langle K.w_n\rangle^{1/2}.$$
 This is a contradiction since $L^2(\Gamma_n\bs G)$ is tempered.

Alternatively, one can use \cite[Theorem 4.2]{Fell} that $L^2(H\ba G)$ is weakly contained in the direct sum $\bigoplus_{n=1}^\infty L^2(\Ga_n\bs G)$. If $\Ga_n$ were all tempered, we would deduce that $L^2(H\ba G)$ is weakly contained in $\bigoplus_{n=1}^\infty L^2(G)$, hence in $L^2(G)$, which then implies that $H$ is tempered.
\end{proof}

\begin{Def}
We say that a sequence of discrete subgroups $\Ga_i$ of $G$  converges to a discrete subgroup $\Ga$ algebraically if
there exists a sequence of isomorphisms $$\chi_i:\Ga \to \Ga_i$$ such that for all $\ga\in \Ga$,
$\chi_i(\ga)$ converges to $\ga $ as $i\to \infty$. In other words, $\chi_i$ converges to the natural inclusion $\op{id}_\Ga$ in $\Hom(\Ga, G)$,  where the space $\Hom(\Ga, G)$ is endowed with the topology of pointwise convergence. In this case, $\Ga$ is called the algebraic limit of $\Ga_i$
\end{Def}

\begin{remark}
We refer the readers to \cite{BS} for a comparison of
 algebraic and Chabauty convergence; in particular, each notion faily to imply the other in general.
\end{remark} 
\begin{theorem} \label{al} The algebraic limit of a sequence of tempered discrete subgroups of $G$ is tempered. 

    \end{theorem}
\begin{proof}
Let $\Ga_i$ be a sequence of tempered discrete subgroups of $G$ which converges to a discrete subgroup $\Ga$ algebraically.
By passing to a subsequence if necessary, we may assume that $\Ga_i$ converges to a closed subgroup $H$  in the Chabauty topology. Since $\Ga$ is the algebraic limit of $\Ga_i$, we have
$$\Ga< H.$$

By Theorem \ref{cl}, $H$ is unimodular and tempered. Since any closed unimodular subgroup of a tempered subgroup is tempered by Lemma \ref{bk}, $\Ga$ is tempered as desired.
\end{proof}

The following is an equivalent formulation of Theorem \ref{al}:
\begin{theorem}\label{al2}
    If a discrete subgroup $\Ga$ is a non-tempered subgroup of $G$, there exists an open neighborhood $\cal O$ of $\op{id}_{\Ga}$
    in $\Hom(\Ga, G)$ such that for any $\sigma\in \cal O$, $\sigma(\Ga)$ is non-tempered.
\end{theorem}

\section{Growth indicator of a lattice of $\SO(n,1)$ as a subgroup of $\SO(n,2)$}\label{sec4}
Let $G=\SO^\circ (n,2)$ for $n\ge 2$.
Consider 
the quadratic form $$Q(x_1, \cdots, x_{n+2}) = x_1x_{n+2}+ x_2x_{n+1} +\sum_{i=3}^n x_i^2.$$
We realize $G$  as the identity component of the following special orthogonal group 
$$\op{SO}(Q)=\{g\in \SL_{n+2}(\br):
Q(g X)= Q(X) \text{ for all $X\in \br^{n+2}$} \}.$$ 
Consider  the diagonal subgroup 
$$A=\{ \text{diag}(e^{t_1}, e^{t_2}, 1\cdots, 1, e^{-t_2}, e^{-t_1}):t_1, t_2\in \br \},$$ 
which is a maximal real split torus of $G$. 
We denote by $\frak g$  the Lie algebra of $G$ and set
$$\fa= \{v=\text{diag}(v_1, v_2, 0, \cdots,0, -v_2, -v_1):v_1, v_2\in \br\}=\log A. $$
For simplicity, we write 
$v=(v_1, v_2, 0, \cdots,0, -v_2, -v_1)$ for an element of $\fa$.
Choose a positive Weyl chamber
\be\label{aplus} \fa^+= \{v=(v_1, v_2, 0, \cdots,0, -v_2, -v_1):v_1\ge v_2\ge 0\}.\ee 

Since $G$ is invariant under the Cartan involution $g\mapsto {g^{-T}}$,
$$K=\{g\in G: g g^T=e\}=G\cap \SO(n+2) $$ is a maximal compact subgroup of $G$
and we have the Cartan decomposition
$G=K (\exp \fa^+)  K$. We denote by $\mu:G\to \fa^+$ the Cartan projection of $G$.

We then have two simple (restricted) roots $\alpha_1$ and $\alpha_2$ for $(\frak g, \fa)$ given by
 $$\alpha_1(v)=v_1-v_2\quad\text{ and }\quad \alpha_2 (v)= v_2\quad\text{ for all $v\in \fa$}. $$

By explicit computation of $\frak g$, we can see that 
the set of all positive roots of $\frak g$ is given by 
$$\Sigma^+(\frak g, \frak a)=\{\alpha_1, \alpha_2, \alpha_1+\alpha_2,\alpha_1+2\alpha_2\}.$$

The direct sum of root subspaces is given by 
$$\Biggl\{\begin{pmatrix} 0 & x & Y_1 & z & 0\\
                  0 & 0 & Y_2 & 0  &-z\\
                    & &       & -Y_1^t & -Y_2^t \\
                    & &       & 0 &-x \\
                    & &        & 0 &0
                    \end{pmatrix} :x, z\in \br, Y_1, Y_2\in \br^{n-2} \Biggr\}$$
                 where the subspaces corresponding to $x\in \br$, $Y_1\in \br^{n-2}$, $Y_2\in \br^{n-2}$, and
                 $z\in \br$ are root subspaces for $\alpha_1$, $\alpha_1+\alpha_2$, $\alpha_2$
and $\alpha_1+2\alpha_2$ respectively. Hence
 the multiplicities are given by 
$$m(\alpha_1)=m(\alpha_1+2\alpha_2)=1$$ and
$$m(\alpha_1+\alpha_2)=m(\alpha_2)=n-2.$$
Since
$ (\alpha_1+\alpha_2)(v)=v_1$ and $(\alpha_1+2\alpha_2)(v)= v_1+v_2$,
the half sum of all positive roots counted with multiplicity is
 \be\label{rrr} \rho (v) = \sum_{\alpha\in \Sigma^+} m(\alpha) \alpha(v) =
 \frac{1}{2} \left( nv_1+ (n-2)v_2\right) \quad\text{for $v\in \fa^+$}.\ee 

\subsection*{Bound on growth indicator for general non-lattice subgroups} 
Recall the definition of the growth indicator of a discrete subgroup of $G$ from \eqref{gi}. For any discrete subgroup $\Ga$ of $G$, 
the growth indicator $\psi_{\Ga}$ is concave and upper-semicontinuous \cite[I.1 Th\'eor\`eme]{Q2}.  Since $\dim \fa^+=2$, it follows that
$\psi_{\Ga}$ is continuous on the limit cone $\cal L_{\Ga}$. 

The quantitative Kazhdan's property $(T)$ of the group $G$ obtained in \cite{Oh} yields the following explicit upper bound:
\begin{prop}\label{lo}
 For any non-lattice discrete subgroup $\Ga$ of $G$,
 we have
 $$\psi_{\Ga}(v)\le (n-1) v_1+(n-2)v_2 \quad\text{for all $v\in \fa^+$}. $$
\end{prop}

\begin{proof}
By \cite[Theorem 7.1]{LO}, 
we have
$$\psi_{\Ga}(v) \le (2\rho - \Theta)(v) \text{ for all $v\in \fa^+$} $$
where
$\Theta$ is the half sum of all roots in a maximal strongly orthogonal system of $\Sigma^+(\frak g, \frak a)$.
Since $\{\alpha_1, \alpha_1+2\alpha_2\}$ is a maximal strongly orthogonal system, we have
$$\Theta(v)= v_1\quad\text{for all $v\in \fa^+$}.$$
Therefore $$(2\rho -\Theta )(v)= (n-1)v_1+(n-2)v_2,$$ proving the claim. \end{proof}

\subsection*{Growth indicator for discrete subgroups of $G$ that are
lattices of $H$}

Let $H=\SO^\circ (n,1)$.
The restriction of the quadratic form $Q$ to the hyperplane $V:=\{x_1=x_{n+2}\}$ yields a quadratic form
$Q_0=Q|_V$ in $(n+1)$ variables. We identify 
$$H=\SO^\circ(n,1)=\{g\in G: g(V)=V\} =\SO^\circ (Q_0).$$

Since $H$ is invariant under the Cartan involution $g\mapsto {g^{-T}}$,
 the intersection $K\cap H$ is a maximal compact subgroup of  $H$.
Denoting by $\frak h$ the Lie algebra of $H$, we have
$$\frak h \cap \frak a =\{ \text{diag}(0, v_2, 0 , \cdots, 0,  -v_2, 0) : v_2\in \br \} .$$
Note that  the Cartan projection $\mu(H)$  is equal to $\fa^+\cap \ker \alpha_2$:
$$\mu(H)=\{v=(v_1, 0, \cdots,0,  -v_1):v_1 \ge 0\}.$$

To see that, apply the Weyl element switching the first two rows (and hence the last two rows) to $\frak h\cap \fa $, resulting in $\{(v_2,0,\ldots,0,-v_2) : v_2\in \br\}=\ker \alpha_2$.

\begin{prop}\label{pg} Let $\Ga<G$ be a discrete subgroup such that
 $\Ga$ is a lattice of $H$.  Then 
 \begin{equation}\psi_\Ga (v)=\begin{cases} (n-1) v_1 &\text{ for $v= (v_1, 0, \cdots, 0, -v_1)$, $v_1\ge 0$}\\ -\infty &\text{for $v\notin \mu(H)$} \end{cases} \end{equation}
In other words, \be\label{upper} \psi_\Ga \le \frac{2(n-1)}{n}\rho \quad\text{ on $\fa^+$}\ee 
with the equality on $\mu(H)$.
\end{prop}

\begin{proof}
    Since $\Ga$ is a lattice of $H$,
the limit cone of $\Ga$ satisfies
$$\mathcal L_{\Ga} =\mu(H)=\fa^+\cap \ker\alpha_2.$$
Hence for $v\notin \mu(H)$, $\psi_\Ga(v)=-\infty$.
Let $\|\cdot\|$ denote the norm on $\fa$ induced from the Riemannian metric on $G/K$.  Since $H/(H\cap K)\subset G/K$ is an isometric embedding,
we have that for all $h\in H$,
$\|\mu(h)\|$ is equal to the Riemannian distance $d_{H/(H\cap K)} (ho, o)$ in $H/(H\cap K)$.
Since $\psi_\Ga$ is independent of the choice of a norm, we may assume that for all $h\in H$,
$\|\mu(h)\|$ is equal to the hyperbolic distance $d_{\bH^n}(ho, o)$ by identifying $H/(H\cap K)\simeq \bH^{n}$, which is equivalent to  $\| (v_1, 0, \cdots, 0, -v_1)\|=|v_1|$.
Since $\Ga<H$ is a lattice,
we have
$$\#\{\ga\in \Ga: d_{\bH^n}(\ga o, o)<T\}\sim C e^{(n-1)T}\quad\text{ as $T\to \infty$}$$
(cf. \cite{DRS}, \cite{EM}).
Hence for $v= (v_1, 0, \cdots, 0, -v_1)$ with $v_1\ge 0$,
$$\psi_\Ga (v)=\|v\| \limsup_{T\to \infty} \frac{\log \#\{\ga\in \Ga: \| \mu(\ga)\| \le T\}}{T}=(n-1)v_1 .$$

Since $\rho (v_1, 0, \cdots, 0, -v_1) = \frac{n}{2} v_1$ by \eqref{rho}, the claim follows.
\end{proof}

\begin{remark}
 Note that the upper bound  \eqref{upper} already follows from Proposition \ref{lo} once we know that $\mathcal L_\Ga\subset \fa_{\alpha_1}$.
 The above proposition shows that that upper bound is optimal for the case at hand.
\end{remark}

\begin{remark} We remark that Proposition \ref{pg} holds in a more general setting:
let $G$ be a connected semisimple real algebraic subgroup with Cartan decomposition $G=KA^+K$ and $H<G$  a connected reductive real algebraic subgroup such that
$H=(K\cap H)(A^+\cap H)(K\cap H)$. Let $\Gamma$ be a lattice of $H$. Then $\psi_\Gamma (v)=2\rho_H (v) $ if $v\in \log (H\cap A^+)$ and $-\infty$ otherwise, where $2\rho_H$ is the sum of all positive roots of $(\op{Lie}(H),\log (H\cap A^+))$. 
\end{remark}

We recall the following criterion on the temperedness of $L^2(\Ga\ba G)$.
\begin{theorem} (\cite{EO}, \cite[Theorem 5.1]{LWW}) \label{lww} For any discrete subgroup $\Ga$ of a connected semisimple real algebraic group $G$, we have
   $$\text{ $\psi_{\Ga} \le \rho$ if and only if  $\Ga$ is a tempered subgroup of $G$}.$$
   Moreover, if $\psi_{\Ga} \le  (1+\eta) \rho$, then $L^2(\Ga\ba G)$ is almost $L^p$
   for $p\le \frac{2}{1-\eta}$.
\end{theorem}
That $L^2(\Ga\ba G)$ is almost $L^p$ means that every matrix coefficient of the quasi-regular representation $L^2(\Ga\ba G)$ is $L^{p+\e}$-integrable for any $\e>0$.
By Theorem \ref{chh}, a discrete subgroup $\Ga$ is $G$-tempered if and only if  $L^2(\Ga\ba G)$ is almost $L^2$.

Since $\psi_\Ga=\frac{2(n-1)}{n}\rho $ on $\mu(H)$ by Proposition \ref{pg},
we obtain the following examples of non-tempered subgroups of $G$:
\begin{cor}\label{ntt}
   Let $\Ga$ be a lattice of $H=\SO^\circ (n,1)$, considered as a subgroup of $G=\SO^\circ(n,2)$.
   Then $$\text{  $\Ga$ is $G$-tempered if and only if $n=2$.}$$
   
   Moreover, for each $n\ge 2$,
   $$\text{ $L^2(\Ga\ba G)$  is almost $L^n$.}$$
\end{cor}

\section{Deformations and non-tempered Zariski-dense examples}\label{def}
Let  $G=\SO^\circ (n,2)$ and $H=\SO^\circ (n,1)=\op{Isom}^+(\bH^n)$.
Let $\Ga$ be a torsion-free uniform lattice of $H$ such that $M=\Ga\ba \bH^n$
is a closed hyperbolic $n$-manifold with a properly embedded totally geodesic hyperplane $S$. 

\begin{Rmk} For any $n\ge 2$, such a $\Ga$ exists, for instance, 
consider a quadratic form $Q_0(x_1, \cdots, x_{n+1})=\sum_{i=1}^n x_i^2 -\sqrt d x_{n+1}^2$
 for a square-free integer $d$. 
 Let $\Ga <\SO(Q_0)\cap \SL_{n+1} (\mathbb Z \sqrt d) $ be a torsion -free subgroup of finite index. Then $\Ga$ 
  is a uniform lattice of $\SO(Q_0)$ \cite{BH}.
 Considering $\SL_{n}$ as a subgroup of $\SL_{n+1}$ embedded as the lower diagonal block subgroup, the intersection
 $\Delta=\G \cap \SL_n$ is a uniform lattice of $
 \SO(Q_0)\cap \SL_n\simeq \SO(n-1,1)$.
Now $M=\Ga \ba \bH^n$ is a closed hyperbolic $n$-manifold with a properly embedded geodesic hyperplane $S=\Delta\ba \bH^{n-1} $.
\end{Rmk}

We may assume that $\Ga\cap \SO(n-1,1)=\Delta$ is a uniform lattice
of $\SO(n-1,1)$ by replacing $\Ga$ by a conjugate if necessary.

 We briefly recall the bending construction of Johnson-Millson \cite{JM}. Their bending was constructed with the ambient group $\op{SL}_{n+2}(\br)$. We use a modification by Kassel \cite[Sec. 6]{Ka} where the bending was done inside $G=\SO^\circ(n,2)$. 
There exists a one-parameter subgroup $a_t\in G$ which centralizes $\SO(n-1,1).$
If $S$ is separating, i.e., $M-S$ is the disjoint union of two connected components
$M_1$ and $M_2$, then $\Ga=\Gamma_1*_{\Delta}\Ga_2$. 
Consider the homomorphism $\sigma_t:\Ga\to G$ given by \begin{equation*}
 \sigma_t(\ga) =\begin{cases} \ga &\text{ for $\ga\in \Gamma_1$}  \\ a_t \ga a_{-t} &\text{ for $\ga\in \Ga_2$.   }\end{cases}
\end{equation*}
Since $a_t$ commutes with $\Delta$, $\sigma_t$ is well-defined.
If $S$ does not separate $M$,
then $\Ga$ is an HNN extension of $\Delta$, and we have a homomorphism $\sigma_t$ defined similarly (cf. \cite[Sec 6.3]{Ka}).

The following Zariski density and discreteness results were obtained in \cite{Ka} and
\cite{G} respectively:
\begin{prop}\label{Ka}  For all sufficiently small $t\ne 0$,  $\sigma_t(\Ga)$ is discrete and
Zariski-dense in $G=\SO^\circ(n,2)$.    
\end{prop}

We now give a proof of Theorem \ref{m1}:
\begin{theorem}\label{m} Let $n\ge 3$. 
  For all sufficiently small $t\ne 0$, the subgroup $\sigma_t(\Ga)$ is a non-tempered, Zariski-dense and discrete subgroup of $G=\SO^\circ(n,2)$.    
\end{theorem}
\begin{proof}
 The subgroup $\Ga$ is a non-tempered subgroup of $G$ for $n\ge 3$ by Corollary \ref{ntt}. 
 Hence the claim follows from Theorem \ref{al2} and Proposition \ref{Ka}.
\end{proof}

\section{Anosov representations and non-temperedness}
In this section, we prove a stronger result than Theorem \ref{m1} using the theory of Anosov representations. We keep the notations for $G=\SO^\circ(n,2)$, $H=\SO^\circ(n,1)$, $\frak a$, $\alpha_1, \alpha_2$ etc  from Section \ref{sec4}.
Let $\Ga$ be a torsion-free uniform lattice of $H$ such that the closed hyperbolic manifold
$\Ga\ba \bH^n$ has a properly embedded totally geodesic hyperplane as in Section \ref{def}.

\begin{Def}\label{anosov} For a non-empty subset $\theta\subset \Pi= \{\alpha_1, \alpha_2\}$,
a finitely generated subgroup $\G_0$ of $G$ is called $\theta$-Anosov if there exists $C>0$ such that
for all $\ga\in \Ga_0$ and $\alpha\in \theta$, we have
 $$\alpha (\mu(\ga))\ge C^{-1} |\ga| -C $$
 where $|\ga|$ denotes the word length of $\ga$ with respect to a fixed finite generating subset of $\Ga_0$.
A $\Pi$-Anosov subgroup is called Borel-Anosov. \end{Def} 
\begin{lem}
The subgroup $\Ga$ is an $\{\alpha_1\}$-Anosov subgroup of $G$.
\end{lem}
\begin{proof} 
Note that $\beta_1:=-\alpha_1$ restricted to $\frak h\cap \fa$ is a simple root of $(\frak h,\frak h\cap \fa) $ with respect to the choice of a positive Weyl chamber
$(\frak h\cap \frak a)^+=\{v=(0, v_2, 0, \cdots,0, -v_2, 0): v_2\ge 0\}  $.
Since $\Ga$ is a uniform lattice of $H$, it is in particular a convex cocompact subgroup of $H$, and hence a $\{\beta_1\}$-Anosov subgroup of $H$ \cite{GW}. Therefore
there exists $C\ge 1$ such that for all $\ga\in \Ga$,
$$\beta_1 (\mu_H(\ga))\ge C^{-1} |\ga| -C$$ where $\mu_H$ denotes the Cartan projection map of $H$.
    Since $$\beta_1\circ \mu_H =\alpha_1 \circ \mu|_H ,$$ it follows that
     $\alpha_1 (\mu(\ga))\ge C^{-1} |\ga| -C$ for all $\ga\in \Ga$. This proves the claim.
\end{proof}

\begin{theorem} \label{fin} Let $n\ge 3$, and $G= \SO^\circ(n,2)$.
    There exists a non-empty open subset $\cal O$ of $\op{Hom}(\Ga, G)$ such that for any $\sigma\in \cal O$, we have
\begin{enumerate}
\item $\sigma$ is injective and discrete;
    \item $\sigma(\Ga)$ is a Zariski-dense $\{\alpha_1\}$-Anosov subgroup of $G$;
    \item $\sigma(\Ga)$ is not $G$-tempered.
\end{enumerate}
\end{theorem}

 By \cite[Proposition 8.2]{AB}, the set of Zariski-dense representations of $\Ga$ forms an open subset of $\Hom(\Ga, G)$, which we know is non-empty by Proposition \ref{Ka}. Moreover, all Anosov representations are discrete with finite kernel and the set of all $\{\alpha_1\}$-Anosov representations forms an open subset in $\Hom(\Ga, G)$ by (\cite{GW}, \cite{KLP}).
Since $\Ga$ is assumed to be torsion-free, Theorem \ref{fin} follows from Theorem \ref{al2} and non-temperedness of $\Ga$. 

In the rest of this section, we will give a different proof of Theorem \ref{fin}(3) using the continuity of limit cones under a small deformation of $\Ga$ and the Anosov property of $\Ga$.

For any discrete subgroup $\Ga_0$ of $G$ and
any linear form $\psi\in \fa^*$ such that $\psi>0$ on $\cal L_{\Ga_0}-\{0\}$, 
    denote by $$\delta_{\psi,\Ga_0}$$ the abscissa of convergence of the series
    $s\mapsto \sum_{\ga\in \Ga_0} e^{-s\psi(\mu(\ga))}$. This is well-defined and $0\le \delta_{\psi, \Ga_0}<\infty$. Since $\rho>0$ on $\fa^+-\{0\}$, $\delta_{\rho, \Ga_0}$
is well-defined for any discrete subgroup $\Ga_0<G$.
Theorem \ref{lww} can be reformulated as follows:
\begin{prop}
  For any discrete subgroup $\Ga_0$ of a connected semisimple real algebraic group $G_0$, we have
     $$\text{ $\delta_{\rho,\Ga_0}\le 1$ if and only if  $\Ga_0$ is $G_0$-tempered} .$$
\end{prop}
\begin{proof}
    By \cite[Theorem 2.5]{KMO}, we have $$\psi_{\Ga_0}\le  \delta_{\rho,\Ga_0} \cdot \rho$$ and
    $\psi_{\Ga_0}(v)=  \delta_{\rho,\Ga_0} \cdot \rho(v)$ for some non-zero $v\in \fa^+$. 
    Therefore  the claim follows from Theorem \ref{lww}.
\end{proof}

 Set
 $$\fa_{\alpha_1}=\ker \alpha_2\quad\text{ and } \quad \fa_{\alpha_1}^+=\fa^+\cap \ker \alpha_2.$$  Let $p_{\alpha_1}:\fa\to \fa_{\alpha_1}$ denote the unique projection invariant under the Weyl element fixing $\fa_{\alpha_1}$ pointwise, which is simply the reflection about $\fa_{\alpha_1}$. The space of linear forms $\fa_{\alpha_1}^*$ can be identified with the set of all linear forms in $\fa^*$ which are invariant under $p_{\alpha_1}$.
  The following follows by combining
  \cite[Proposition 8.1]{BCLS} and \cite[Corollary 5.5.3]{Sa}, both of whose proofs are based on thermodynamic formalism.
\begin{theorem}   \label{bcls} 
   For any $\psi\in \fa_{\alpha_1}^*$ which is positive on $\fa_{\alpha_1}^+-\{0\}$,
  the critical exponent $\delta_{\psi, \sigma(\Ga)}$ varies analytically on any
     sufficiently small analytic neighborhood of an $\{\alpha_1\}$-Anosov representation of $\Hom(\Ga, G)$. 
  \end{theorem}

Since $\Ga$ is a convex cocompact subgroup of $H$, the following is a special case of Kassel's theorem
\cite[Proposition 5.1]{Ka} (see also \cite[Theorem 1.1]{DO} for a recent generalization):
\begin{prop} \label{ka} For any $\eta>0$, we have an open neighborhood $\cal O$ of $\op{id}_\Ga$ in $\op{Hom}(\Ga, G)$
such that for any $\sigma\in \cal O$, the limit cone of $\sigma(\Ga)$ is
contained in ${\mathcal C}_\eta:=\{ v\in \fa^+: \|v-\fa_{\alpha_1}\| < \eta\|v\|\}$. 
\end{prop}

\begin{Rmk} For the bending deformations $\sigma_t$ discussed in section \ref{def},
 we always have a non-trivial element of $\ga$ (of infinite order) such that $\sigma_t(\ga)=\ga$, and hence $\mu(\sigma_t(\ga))\in\mu(H) -\{0\}$. Therefore we have the following property:
for all sufficiently small $t\ne 0$,  the limit cone of $\sigma_t(\Ga)$
contains the ray $\mu(H)$.  Since $\sigma_t(\Ga)$ is Zariski-dense, its limit cone is convex and has non-empty interior \cite{Ben97}. Therefore  
Proposition \ref{ka}  implies that 
that the limit cone of $\sigma_t(\Ga)$ is the convex cone given
\be\label{expl} \cal L_{\sigma_t(\Ga)} =\{ v=(v_1, v_2, 0, \cdots,-v_2, -v_1)\in \fa^+: 0\le v_2 \le c_{\sigma_t}  v_1\}\ee 
where $c_{\sigma_t} >0$ tends to $0$ as $t\to 0$.\end{Rmk}

Recall from Proposition \ref{pg}. that $$\delta_{\rho,\Ga}= \frac{2(n-1)}{n}.$$
The following proposition gives an alternative proof of Theorem \ref{fin}(3):
\begin{prop}\label{alter} For any sufficiently small $\epsilon>0$, there exists an open neighborhood $\cal O=\cal O(\e) $ of $\op{id}_\Ga$ in $\op{Hom}(\Ga, G)$ such that
for any $\sigma\in \cal O$,
$$\left|  \delta_{\rho, \sigma(\Ga)}  - \frac{2(n-1)}{n}\right| < \e  .$$

In particular, for $n\ge 3$, we have $ \psi_\Ga\not\le \rho ;$ and hence
 $\sigma(\Ga)$ is non-tempered in $G$ for all $\sigma\in \cal O(\frac{n-2}{n})$
\end{prop}
\begin{proof} 
Let $\rho'$ be the restriction of $\rho$ to $\fa_{\alpha_1}$. We may consider $\rho'$ as a linear form on $\fa$  by precomposing with $p_{\alpha_1}$. Note that $\rho'$ is non-negative  on $\fa_{\alpha_1}^+$.

Let $\e>0$.
We can find $\eta>0$ so that for any $v\in {\mathcal C}_\eta=\{ v\in \fa^+: \|v-\fa_{\alpha_1} \| < \eta\|v\|\}$,
$$-\e\rho(v)\le  (\rho - \rho')(v)  \le \e \rho(v).$$

We can take a small neighborhood $\cal O$ of $\op{id}_\Ga$ so that for any $\sigma\in \cal O$, the limit cone of $\sigma(\Ga)$ is contained in the cone ${\mathcal C}_\eta$ by Proposition \ref{ka}.
In particular, $\mu(\sigma(\ga))\in  \mathcal C_\eta$
for all $\ga\in \G$ except for some finite subset $F_\sigma$.
Then for any $\sigma\in \cal O$, we have that for all $s>0$,
$$\sum_{\ga\in \Ga-F_\sigma } e^{-(1-\epsilon) s\rho (\mu(\sigma(\ga)))} \ge \sum_{\ga\in \Ga-F_\sigma} e^{- s\rho' (\mu(\sigma(\ga)))}.  $$
It follows that  $$\delta_{(1-\epsilon) \rho,\sigma(\Ga) } \ge  \delta_{\rho', \sigma(\Ga)}\;\; \text{ and hence }\;\;
\delta_{\rho,\sigma(\Ga)}  \ge (1-\epsilon)  \delta_{\rho',\sigma(\Ga)}. $$
Similarly, 
we have
$$\sum_{\ga\in \Ga-F_\sigma} e^{-(1+\epsilon) s\rho (\mu(\sigma(\ga)))} \le \sum_{\ga\in \Ga-F_\sigma} e^{- s \rho'(\mu(\sigma(\ga)))}, $$
 $$\delta_{(1+\epsilon) \rho,\sigma(\Ga) } \le  \delta_{\rho', \sigma(\Ga)}\;\; \text{ and hence }\;\;
\delta_{\rho,\sigma(\Ga)}  \le (1+\epsilon)  \delta_{\rho',\sigma(\Ga)}. $$

Therefore
\be\label{f}  (1-\e) \delta_{\rho', \sigma(\Ga)} \le \delta_{\rho,\sigma(\Ga)} \le (1+\e) \delta_{\rho',\sigma(\Ga)}.\ee 

By replacing $\cal O$ by a smaller  neighborhood of $\op{id}_\Ga$ if necessary,  we may assume that 
\be\label{f2} | \delta_{\rho', \sigma(\Ga)} -\delta_{\rho', \Ga}|\le \e \quad\text{ for all $\sigma\in \cal O$ } \ee by Theorem \ref{bcls}. 

Hence using that $1\le \delta_{\rho, \Ga}=2(n-1)/n\le 2$,
we deduce from \eqref{f} and \eqref{f2} that
$$|\delta_{\rho, \sigma(\Ga)} -\delta_{\rho, \Ga} |<5 \e 
\quad\text{ for all $\sigma\in \cal O$ } .$$

Since
$\delta_{\rho, \Ga}=2(n-1)/n$, the claim follows.
\end{proof}

We can also obtain the following estimates for  the growth indicator $\psi_{\sigma(\Ga)}$:
\begin{cor} \label{c5} For any sufficiently small $\e>0$, there exists an open neighborhood $\cal O=\cal O(\e) $ of $\op{id}_\Ga$ in $\op{Hom}(\Ga, G)$ such that
for any $\sigma\in \cal O$,
$$\psi_{\sigma(\Ga)}(v)\le 
\left( \frac{2(n-1)}{n} +\e\right ) \rho (v)   \quad\text{for all $v\in \fa^+$} $$
and
    \be\label{vs} \psi_{\sigma(\Ga)}(v_\sigma)\ge \left( \frac{2(n-1)}{n} -\e\right) \rho (v_\sigma) \quad\text{for some
  unit vector $v_\sigma \in \fa^+$}.\ee 
Moreover, $v_{\sigma}$ converges to a unit vector in $\fa_{\alpha_1}$ as $\sigma\to \id_\Ga$.    
 \end{cor}
\begin{proof}
Recall that $\psi_{\sigma(\Ga)}\le \delta_{\rho, \sigma(\Ga)}\rho$ and
$\psi_{\sigma(\Ga)}(v_\sigma)=\delta_{\rho, \sigma(\Ga)}\rho(v_\sigma)$ for some non-zero vector $v_\sigma$ on the limit cone $\cal L_{\sigma(\Ga)}$ \cite[Theorem 2.5]{KMO}.
Hence the inequalities follow from Proposition \ref{alter}. The last claim follows from Proposition \ref{ka}.
\end{proof}

Finally, since $v_\sigma$ is of the form
$ (v_{\sigma,1}, c_\sigma v_{\sigma_,1}, 0, \cdots, -c_\sigma v_{\sigma,1}, -v_{\sigma,1})$ for some 
$v_{\sigma,1}>0$ with $c_\sigma\to 0$, the inequality
\eqref{vs} and Proposition \ref{lo} imply  the inequality \eqref{vs0} in Theorem \ref{mm}. Hence,
together with Theorem \ref{fin}, Proposition \ref{alter} and Corollary \ref{c5}, this completes the proof of Theorem \ref{mm}.

\end{document}